\documentclass{article}

\usepackage{amsmath}
\usepackage{amssymb}
\usepackage{amsthm}
\usepackage{graphicx,color}
\usepackage{texdraw}
\newtheoremstyle{plain2}{\topsep}{\topsep}%
     {\itshape}
     {}
     {\bfseries}
     {.}
     {.5em}
     {\thmnumber{(#2)}\thmname{ #1}\thmnote{ #3}}

\theoremstyle{plain2}
\newtheorem{teo}{Theorem}[section]
\newtheorem{prop}[teo]{Proposition}
\newtheorem{coro}[teo]{Corollary}
\newtheorem{lemma}[teo]{Lemma}
\newtheorem*{quest}{Question}

\newtheoremstyle{definition2}{\topsep}{\topsep}%
     {}
     {}
     {\bfseries}
     {.}
     {.5em}
     {\thmnumber{(#2)}\thmname{ #1}\thmnote{ #3}}

\theoremstyle{definition2}

\newtheorem{rem}[teo]{Remark}
\newtheorem{defi}[teo]{Definition}

\def\N{\mathbb{N}}

\def\R{\mathbb{R}}
\def\C{\mathbb{C}}

\def\b{\beta}
\def\g{\gamma}
\def\d{\delta}

\def\l{\lambda}

\def\s{\sigma}

\def\D{\Delta}

\def\O{\Omega}

\def\invD{\mathcal{L}}

\def\tl{{\tilde\lambda}}
\def\tmu{{\tilde\mu}}

\def\tS{{S_\infty}}

\def\Cah{\mathcal{C}}

\def\Kah{\mathcal{K}}
\def\Feh{\mathcal{F}}
\def\Mah{\mathcal{M}}

\def\spc{H^1_0(\O)}

\def\dist{\mathrm{dist}}

\title{\sc Convergence of minimax and\\ continuation of critical points\\ for singularly perturbed systems.}

\author{%
Benedetta Noris, Hugo Tavares, Susanna Terracini and Gianmaria Verzini}

\begin{document}

\maketitle

\begin{abstract}
We consider the system of stationary Gross--Pitaevskii equations
\[
\left\{
  \begin{array}{l}
   -\D u + u^3+\b uv^2=\l u\smallskip\\
   -\D v + v^3+\b u^2v=\mu v\smallskip\\
   u,v\in\spc,\quad u,v>0,
 \end{array}
\right.
\]
arising in the theory of Bose--Einstein condensation, and the related scalar equation
\[
 -\D w + w^3=\l w^+-\mu w^-.
\]
We address the following
\begin{quest}
Is it true that every bounded family $(u_\beta,v_\beta)$ of solutions of the system converges, as $\b\to+\infty$, up to a subsequence, to a pair $(u_\infty,v_\infty)$, where $u_\infty-v_\infty$ solves the scalar equation?
\end{quest}
We discuss this question in the case when the solutions to the system are obtained as  minimax critical points via (weak) $L^2$ Krasnoselskii genus theory. Our results, though still partial, give a strong indication of a positive answer.
\end{abstract}

\section{Introduction}

\subsection{Motivations}

The commonly proposed mathematical model for binary
mixtures of Bose--Einstein condensates is a system of two coupled
nonlinear Schr\"odinger equations, known in the literature as Gross--Pitaevskii equations, which writes
\[
\left\{
 \begin{array}{l}
  \displaystyle -\imath \partial_t\phi=\Delta \phi-V_1(x)\phi- \mu_1|\phi|^2\phi-
  \b_{12}|\psi|^2\phi\smallskip\\
  \displaystyle -\imath \partial_t\psi=\Delta \psi-V_2(x)\psi- \mu_2|\psi|^2\psi-
  \b_{21}|\phi|^2\psi\smallskip\\
  \phi,\,\psi \in H^1_0(\O;\C),
 \end{array}
\right.
\]
with $\O$ a smooth bounded domain of $\R^N$, $N=2,3$
(see \cite{clll} and references therein). Here the complex valued functions $\phi,\psi$ are the
wave functions of the two condensates, the real functions $V_i$ ($i=1,2$)
represent the trapping magnetic potentials, and the positive
constants $\b_{ij}$ and $\mu_i$  are the interspecies and the
intraspecies scattering lengths, respectively. With this choice both
the interactions between unlike particles and the interactions between
the like ones are repulsive (the so called \emph{defocusing} case, opposed
to the \emph{focusing} one, where the $\mu_i$'s are negative). When searching for solitary wave solutions
of the above system, one makes the ansatz
\[
\phi(t,x)=e^{-\imath\l t}u(x),\quad \psi(t,x)=e^{-\imath\mu t}v(x),
\]
obtaining that the real functions $u$, $v$ satisfy the system
\begin{equation}\label{eq:sys}
\left\{
 \begin{array}{l}
  S_{\b}(u,v;\l,\mu)=\left(
  \begin{array}{l}
   -\D u + u^3+\b uv^2-\l u\smallskip\\
   -\D v + v^3+\b u^2v-\mu v\smallskip\\
  \end{array}
   \right) =
  0\\
   u,v\in\spc,\quad u,v>0,
 \end{array}
\right.
\end{equation}
at least when $V_i(x)\equiv0$, $\mu_1=\mu_2=1$, and $\b_{12}=\b_{21}=\b$
(in fact, we can treat more general systems, provided
the problem is symmetric with respect to $u$ and $v$ and it is of variational type).
In this paper we are interested in the relation between suitable solutions
$(u_\b,v_\b)$ of \eqref{eq:sys}, for $\b$ large, and the
pairs\footnote{Here, as usual, $w^\pm(x)=\max\left\{\pm
w(x),0\right\}$ denote the positive and negative part of a function
$w$.} $(w^+,w^-)$, where $w$ solves
\begin{equation}\label{eq:eq}
-\D w + w^3=\l w^+-\mu w^-, \qquad w\in\spc
\end{equation}
(for suitable $\l$, $\mu$), which represents, as we shall see, a limiting problem for \eqref{eq:sys} as $\b\to+\infty$.

The standard mass conservation law gives $\int_\O u^2=m_1$,
$\int_\O v^2=m_2$. Assuming $m_{1}=m_{2}=1$ we are led to study solutions of \eqref{eq:sys} as (nonnegative) critical points
of the coercive energy functional
\begin{equation*}
J_\b(u,v)=\frac12\left(\|u\|^2+\|v\|^2\right)+\frac14\int_{\O}
\left(u^4+v^4\right)\,dx+\frac\b2\int_{\O}u^2v^2\,dx
\end{equation*}
constrained to the manifold
\[
M=\left\{(u,v)\in\spc\times\spc:\,\int_\O u^2 dx=\int_\O v^2 dx=1\right\},
\]
so that $\l$ and $\mu$ in \eqref{eq:sys} can be understood as Lagrange multipliers. The natural
limiting energy, at least for the minimization purposes, is of course the $\Gamma$--limit of the family $J_\b$, that is,
the extended valued functional defined as\footnote{In this definition, of course, one can equivalently use any $J_\b$, $0<\b<+\infty$, instead of $J_0$.}
\[
J_\infty(u,v)=\sup_{\b>0}J_\b(u,v)=
\begin{cases}
J_0(u,v) & \text{when }\int_{\O}u^2v^2\,dx=0\smallskip\\
+\infty  & \text{otherwise}
\end{cases}
\]
(for the definition of $\Gamma$--convergence we refer, for instance, to the book by Braides \cite{braides}). This
functional is far from being $C^1$, indeed it is finite only on a non smooth domain. On the other hand, when finite, we have that
\begin{equation}\label{eq:J*}
J_\infty(u,v)=J^*(u-v),\qquad\text{where }J^*(w)=\frac12\|w\|^2+\frac14\int_{\O}w^4\,dx
\end{equation}
(with $\int_\O (w^+)^2=\int_\O (w^-)^2=1$). This suggests to understand the critical points of $J_\infty$ constrained to $M$ as the (nonnegative)
pairs $(u,v)$, such that
$u\cdot v\equiv0$ and $u-v$ satisfies equation \eqref{eq:eq} (and again $\l$ and $\mu$ play the role
of Lagrange multipliers).

One can easily find the so called ground state solutions of the functionals (both for $\beta$ finite and infinite), that is, the constrained minimizers of the energies (moreover, each component of the ground state can be chosen to be nonnegative). For such solutions it is an easy exercise, at least in the defocusing case, to  prove convergence. Indeed, for any sequences
$\b_n \to +\infty$ and $(u_n,v_n)$ minimizers of $J_{\b_n}$, there exists
a pair $(u_\infty,v_\infty)$, minimizer of $J_\infty$, such that, up to subsequences,
\(
(u_n,v_n)\to(u_\infty,v_\infty)
\),
strongly in $H^1_0$ (this fact can be easily read in the framework of $\Gamma$--convergence theory).
Moreover, the convergence can be proven to be more regular, in particular $C^{0,\alpha}(\overline{\O})$, see \cite{nttv1}.
The remarkable fact is that any $(u_\infty,v_\infty)$, minimizer of $J_\infty$, corresponds to a minimizer
$w=u_\infty-v_\infty$ of the smooth functional $J^*$ (with the appropriate constraint); this provides a suitable differential extremality condition
for the minimizers of the non--smooth functional $J_\infty$, in the form of equation \eqref{eq:eq}.
\begin{teo}\label{teo:ground_states}
Let $(u_\b,v_\b)\in M$, for $\b\in(0,+\infty)$, be a minimizer of $J_\b$ constrained to $M$. Then, up to subsequences, $(u_\b,v_\b)\to(u_\infty,v_\infty)$, strongly in $H^1\cap C^{0,\alpha}$, minimizer of $J_\infty$
constrained to $M$. Moreover $u_\infty-v_\infty$ solves \eqref{eq:eq}.
\end{teo}
We shall obtain a proof of this result, as a byproduct of a more general one, at the end of Section \ref{sec:main1_beta}. The situation becomes much more delicate when coming to the excited states, i.e. solutions that are not minimal.
The existence of such kind of solutions has been recently proved, mainly in the focusing case, in \cite{dww,nr,TV,ww}.
Moreover, in the aforementioned \cite{nttv1}, the authors have obtained  estimates in H\"older spaces, uniform with respect to $\beta$. In light of the analysis for the ground states, we formulate the following conjecture.

\noindent\textbf{Conjecture. }
\emph{Every bounded family $(u_\beta,v_\beta)$ of solutions to \eqref{eq:sys} converges, as $\b\to+\infty$, up to a subsequence, to a pair $(u_\infty,v_\infty)$, where $u_\infty-v_\infty$ solves \eqref{eq:eq}.}

Unfortunately, we are still not able to prove this conjecture. However, in this paper we establish some limiting relations between critical levels and optimal sets as $\b\to+\infty$, with respect to a common minimax variational framework (see the theorems below). When dealing with minimax critical points, a key role is played by the gradient flow; one of the main difficulties in our situation relies on the fact that we will have to compare a family of gradient flows (at $\b<+\infty$) with a limiting one, induced on the functional $J_\infty$ by $J^*$. These flows do not seem to be related, nor the family of flows seems to converge in a sufficiently strong way (see \cite{caflin2}). This prevents us to apply the recent theory of $\Gamma$--convergence of gradient flows developed in \cite{sandserfaty}.

\subsection{A class of minimax problems}

To proceed with the exposition of our main results, we need to introduce some suitable minimax framework which is admissible for the whole family of functionals. In doing this, we are inspired by a recent work by Dancer, Wei and Weth \cite{dww}, where infinitely many
critical levels are found, in the focusing case, by exploiting the Krasnoselskii genus technique (see, for instance, the book by Struwe \cite{struwe}) associated with the invariance of the problem when interchanging the role of $u$ and $v$.

In carrying on our asymptotic analysis, we shall take advantage of a strong compactness property that goes beyond the usual Palais--Smale condition; to this aim we are lead to set the genus theory in the $L^{2}$--topology. This is the main reason why we are addressing here the defocusing case: in the focusing one, indeed, the fact that the associated Nehari manifold is not $L^2$--closed seems to prevent us to perform an analogous analysis. Let us consider the involution
\[
\s: \spc\times\spc \to \spc\times\spc,\quad (u,v)\mapsto \s(u,v)=(v,u),
\]
and the class of sets
\[
\mathcal{F}_0=\left\{A\subset M:\,\begin{array}{l}
\bullet\ A \text{ is closed in the $L^2$--topology}\\
\bullet\ (u,v)\in A\implies u\geq0,v\geq0\\
\bullet\ \sigma(A)=A
\end{array}  \right\}
\]
(observe that $M$ is $L^2$--closed and that $\s(M)=M$). We can define the Krasnoselskii
$L^2$--genus in $\mathcal{F}_0$ in the following way.
\begin{defi}\label{defi:L2_genus}
Let $A\in\mathcal{F}_0$. The \emph{$L^2$--genus} of $A$, denoted by
$\g_2(A)$, is defined as
\[
\g_2(A)=\inf\left\{m\in\N:\,
\begin{array}{l}
 \text{there exists }f:A\to\R^m\setminus\{0\}\text{ such that}\\
 \quad \bullet\ f\text{ is continuous in the $L^2$--topology and}\\
 \quad \bullet\ f(\sigma(u,v))=-f(u,v)\text{ for every }(u,v)\in A
\end{array}\right\}.
\]
If no $f$ as above exists, then $\g_2(A)=+\infty$, while $\g_2(\emptyset)=0$. The set of subsets
with $L^2$--genus at least $k$ will be denoted by
\[
\mathcal{F}_k=\left\{A\in\mathcal{F}_0:\,\g_2(A)\geq k\right\}.
\]
\end{defi}
Under the previous notations we define, for $0<\b\leq+\infty$, the (candidate) critical levels
\begin{equation}\label{eq:lev_set}
c^k_\b=\inf_{A\in\mathcal{F}_k}\sup_{(u,v)\in A} J_\b(u,v).
\end{equation}
When $\b<+\infty$, the (candidate) critical set is defined in the standard way:
\[
 \mathcal{K}^k_{\b}=\left\{(u,v)\in M:
 \begin{array}{l}
   u,v\geq0,\\
   J_\b(u,v)=c^k_\b,\text{ and}\\
   \text{there exist }\l,\mu\text{ such that }S_\b(u,v;\l,\mu)=(0,0)
 \end{array}
 \right\}.
\]
Coming to the limiting problem, inspired by Theorem \ref{teo:ground_states}, we define the critical set as
\[
\mathcal{K}^k_{\infty}=\left\{(u,v)\in M:
 \begin{array}{l}
   u,v\geq0,\\
   J_\infty(u,v)=c^k_\infty,\text{ and}\\
   \text{there exist }\l,\mu\text{ such that }u-v\text{ solves \eqref{eq:eq}}
 \end{array}
 \right\}.
\]
Our first main result is the following.
\begin{teo}\label{thm:main1}
Let $k\in\N^+$ and $0<\b\leq+\infty$ be fixed. Then:
\begin{enumerate}
 \item $ \mathcal{K}^k_\b$  is non empty and compact (with respect to the $H^1_0$--topology);
 \item there exists $A^{k}_\b\in\mathcal{F}_k$ and $(u^{k}_\b,v^{k}_\b)\in A^{k}_\b\cap\mathcal{K}^k_\b$ such that
 \[
 c^k_\b=\max_{A^{k}_\b}J_\b=J_\b(u^{k}_\b,v^{k}_\b).
 \]
\end{enumerate}
\end{teo}
As in the usual genus theory, one may also prove that, if $c^k_\b$ is the same for different $k$'s, then the genus
of $\Kah_\b^k$ is large. This, together with suitable conditions which allow to avoid fixed points of $\sigma$ (namely $\beta$ large enough, see Lemma \ref{lem:no_fix_point}), provides the existence of many distinct critical points of $J_\b$.

\subsection{Limits as $\b\to+\infty$.}

Since the same variational argument applies both to the $\b$--finite and to the limiting case, the next step is to compare the limiting behaviour of the variational structure as $\b\to+\infty$ with the actual behaviour at
$\b=+\infty$. This involves the study of the critical levels, of the optimal sets (in the sense of Theorem
\ref{thm:main1}) and, finally,  of the critical sets. Regarding the first two questions, we have full convergence.
\begin{teo}\label{thm:main2}
Let $k\in\N^+$ be fixed. As $\b\to+\infty$ we have
\begin{enumerate}
 \item $c^{k}_{\b}\to c^{k}_{\infty}$;
 \item if $A^{k}_n$ is any optimal set for $c^{k}_{\b_n}$, and $\b_n\to+\infty$,
       then the set $\limsup_n A^{k}_n$ is optimal for $c^{k}_{\infty}$ (the limit is intended in the
       $L^2$--sense).
\end{enumerate}
\end{teo}
It is worthwhile to notice that, in general, the convergence of the critical levels is a delicate fact to prove (for instance, it remains an open problem in \cite{nr}). Coming to the convergence of the critical sets, were our conjecture true, we would obtain that
\[
\Kah^k_\infty \supset\limsup_{\b\to+\infty} \Kah^k_\b :
=\left\{\begin{array}{l}(u,v):\text{ $\exists$ sequences $(u_n,v_n)\in M$, nonnegative, $\b_n\to+\infty$ with}\\
   \qquad \bullet\ (u_n,v_n)\to (u,v)\text{ in }L^2,\\
   \qquad \bullet\ J_{\b_n}(u_n,v_n) = c^k_\b\to c_\infty^k,\text{ and}\\
   \qquad \bullet\ S_{\b_n}(u_n,v_n)= (0,0)
 \end{array}
 \right\}.
\]
As we already mentioned, we obtain a weaker result. As a matter of fact, it is more likely, in the framework of Bose--Einstein condensation theory, to prove convergence of the critical levels rather than convergence of the critical points (see e.g. \cite{aftahelf}). The same problem arises in the framework of $\Gamma$--convergence as pointed out in the recent paper \cite{jerrstern}. Here the authors prove, under strong nondegeneracy assumptions, that the existence of critical points for a $\Gamma$--limit may provide the existence of critical points for the approximating functionals, but in general there is no reason why the critical points should be close. Our result is the following.
\begin{teo}\label{thm:main3}
Let
\[
\Kah_*^k=\left\{\begin{array}{l}(u,v):\text{ $\exists$ sequences $(u_n,v_n)\in M$, nonnegative, $\b_n\to+\infty$ with}\\
   \qquad \bullet\ (u_n,v_n)\to (u,v)\text{ in }L^2,\\
   \qquad \bullet\ J_{\b_n}(u_n,v_n) \to c^k_\infty,\text{ and}\\
   \qquad \bullet\ S_{\b_n}(u_n,v_n)\to (0,0)\text{ in }L^2
 \end{array}
 \right\}.
\]
Then
\[
\Kah_*^k\cap\Kah_\infty^k\quad\text{is not empty}.
\]
\end{teo}
Using the uniform H\"older bounds obtained in \cite{nttv1}, we obtain that the $L^2$--convergences in the definitions of $\limsup_\b \Kah_\b^k$, $\Kah^k_*$, are in fact strong in $H^1\cap C^{0,\alpha}$. In particular this implies, in the direction of our conjecture, the following result.
\begin{coro}\label{coro:intro_last}
For every integer $k$ there exist pairs $(u_\infty,v_\infty)$, $(\l_\infty,\mu_\infty)$ satisfying
\[
-\D (u_\infty-v_\infty) + (u_\infty-v_\infty)^3 + \l_\infty u_\infty - \mu_\infty v_\infty=0,
\]
at level $c^k_\infty$, and (sub)sequences $(u_\b,v_\b)$, $(f_\b,g_\b)$, $(\l_\infty,\mu_\infty)$ satisfying
\[
\left\{
   \begin{array}{l}
   -\D u_\b + u_\b^3+\b u_\b v_\b^2-\l_\b u_\b=h_\b\smallskip\\
   -\D v_\b + v_\b^3+\b u_\b^2v_\b-\mu_\b v_\b=k_\b\smallskip\\
   u_\b,v_\b\in\spc,\quad u,v>0,
 \end{array}
 \right.
 \]
such that $(\l_\b,\mu_\b)\to(\l_\infty,\mu_\infty)$,
\[
(h_\b,k_\b)\to(0,0)\text{ in }L^2,\quad\text{ and }\quad (u_\b,v_\b)\to(u_\infty,v_\infty)\text{ in }H^1\cap C^{0,\alpha}.
\]
\end{coro}
The paper is structured as follows. In Section \ref{sec:abstrac} we present an abstract framework of variational type; we introduce a family of functionals enjoying suitable properties and perform an asymptotic analysis. Section \ref{sec:conv_lev} is devoted to fit the problem that we just presented into the abstract setting; this immediately provides the convergence of the critical levels and of the optimal sets. Finally, in Section \ref{sec:main1_beta}, we conclude the proof of the main results: we address existence and asymptotics of the critical points, leaving to Section \ref{sec:construction_fluxes} the technical details about the flows used in the deformation lemmas.
\textbf{Notations.} In the following $\|u\|^2=\int_\O|\nabla u|^2\,dx$, $|u|_p^p=\int_\O u^p\,dx$ (sometimes it will also denote the vectorial norm). We will
refer to the topology induced on $\spc\times\spc$ by the $L^2(\O)\times L^2(\O)$ norm as the ``$L^2$--topology'' (and we shall denote by $\langle\cdot,\cdot\rangle_2$, $\dist_2$ the associated scalar product and distance respectively). On the
other hand, we will denote the usual topology on $\spc\times\spc$ the ``$H^1_0$--topology''. Finally, recall that,
for a sequence of sets $(A_n)_n$,
\[
x\in\limsup_nA_n \iff \text{for some }n_k\to+\infty\text{ there exist }x_{n_k}\in A_{n_k}\text{ such that }x_{n_k}\to x.\footnote{This is the limit superior in the framework of the Kuratowski convergence.}
\]
%
\section{Topological setting of a class of minimax principles}\label{sec:abstrac}
In this section we will introduce an abstract setting of min--max type in order to obtain critical values (in a suitable sense) of a given functional. Our aim is to consider a class of functionals, each of which fitting in such a setting, and to perform an asymptotic analysis of the variational structure. The asymptotic convergence requires some additional compactness, in the form of assumptions ($\Feh$2), ($\Feh$2') below. Later on, when applying these results, this will be achieved by means of weakening the topology; the price to pay will be a loss of regularity of the functional involved. For this reason, with respect to the usual variational schemes, our main address is to work with functionals that are only lower semi-continuous.

Let $(\Mah,\dist)$ be a metric space and let us consider a set of subsets of $\Mah$,
$\Feh\subset 2^\Mah$. Given a lower semi-continuous functional $J:\Mah\to\R\cup\{+\infty\}$, we define the min--max level
\[
c=\inf_{A\in\Feh} \sup_{x\in A} J(x),
\]
and make the following assumptions:
\begin{itemize}
\item[($\Feh$1)] $A$ is closed in $\Mah$ for every $A\in \Feh$;
\item[($\Feh$2)] there exists $c'>c$ such that for any given $(A_n)_n\subset\Feh$ with $A_n\subset \Mah^{c'}$ for every $n$, it holds $\limsup_n A_n\in \Feh$,
\end{itemize}
where
\[
\Mah^{c'}=\left\{x\in\Mah:\,J(x)\leq c'\right\}.
\]

Moreover from now on we will suppose that $c\in \R$, which in particular implies that $\Feh\neq \emptyset$ and $\emptyset\notin \Feh$. A first consequence of the compactness assumption ($\Feh$2) is the existence of an optimal set of the minimax procedure.
\begin{prop}\label{prop:abstr_1}
Let $J:\Mah\to\R\cup\{+\infty\}$ be a lower semi-continuous functional, assume ($\Feh$2) and suppose moreover that $c\in \R$. Then there exists $\bar A\in\Feh$ such that $\sup_{\bar A} J=c$. In this situation, we will say that $\bar A$ is optimal for $J$ at $c$.
\end{prop}
\begin{proof}
For every $n\in \N$ let $A_n\in\Feh$ be such that
\[
\sup_{A_n}J\leq c+\frac{1}{n}
\]
and consider $\bar A:=\limsup_n A_n$. On one hand $\bar A\in  \Feh$ by assumption ($\Feh$2) which provides $\sup_{\bar A}J\geq c$. On the other hand, by the definition of $\limsup$, for any $x\in \bar A$ there exists a sequence $(x_n)_n$, $x_n\in A_n$, such that, up to a subsequence, $x_n\to x$. But the lower semi-continuity implies
\[
J(x)\leq\liminf_n J(x_n)\leq \liminf_n \left(\sup_{A_n}J \right)\leq c,
\]
and the proposition follows by taking the supremum in $x\in\bar A$.
\end{proof}

Due to the lack of regularity of the functional it is not obvious what should be understood as critical set. We will give a very general definition of critical set at level $c$ by means of a ``deformation'', defined in some sub-level of $J$, under which the functional decreases. To this aim we consider, for some $c'>c$, a map $\eta: \Mah^{c'}\to\Mah^{c'}$ such that
\begin{itemize}
 \item[($\eta$1)]  $\eta(A)\in\Feh$ whenever $A\in\Feh$, $A\subset\Mah^{c'}$;
 \item[($\eta$2)] $J(\eta(x))\leq J(x)$, for every $x\in\Mah^{c'}$.
\end{itemize}
We define the critical set of $J$ (relative to $\eta$) at level $c$ as
\[
\Kah_c=\left\{x\in\Mah:\,J(x)=J(\eta(x))=c\right\}
\]
(notice that $x\in\Mah^c$ and hence $\eta(x)$ is well defined). We remark that the previous definition depends on the choice of $\eta$. In a quite standard way, some more compactness is needed in the form of a Palais--Smale type assumption.
\begin{defi}\label{def:palais_smale_abstrac}
We say that the pair ($J,\eta$) satisfies $(PS)_c$ if for any given sequence $(x_n)_n\subset\Mah$ such that $J(x_n)\to c$, $J(\eta(x_n))\to c$, there exists $\bar x\in \Kah_c$ such that, up to a subsequence, $x_n\to\bar x$ (as above, $\eta(x_n)$ is well defined for $n$ sufficiently large).
\end{defi}

\begin{rem}\label{rem:J_const=>eta_fix}
Incidentally we observe that, if in (PS$)_c$ one would require $\bar x$ to be also the limit of $\eta(x_n)$ (we do not assume it in this section, but it will turn out to be true in the subsequent application), as a
consequence $\Kah_c$ would coincide with the set of the fixed points of $\eta$ at level $c$, providing an alternative definition -- probably more intuitive -- of ``critical set'' (relative to $\eta$).
\end{rem}

As usual, (PS$)_c$ immediately implies the compactness of $\Kah_c$. This assumption also implies the fact that every optimal set for $J$ at level $c$ (recall Proposition \ref{prop:abstr_1}) intersects $\Kah_c$ (which in particular is non empty). More precisely
\begin{teo}\label{thm:abstr_1}
Let $J:\Mah\to\R\cup\{+\infty\}$ be a lower semi-continuous functional, assume ($\Feh 1$) and ($\Feh 2$) and let $\eta:\Mah^{c'}\rightarrow \Mah^{c'}$ be a map such that ($\eta1$) and ($\eta2$) hold. Suppose moreover that $(J,\eta)$ verify $(PS)_c$ and that $c\in\R$. Then for every $A\in\Feh$ such that $\sup_{A} J=c$ there exists $\bar x\in A\cap\Kah_c$. In particular, $K_c$ is non empty.
\end{teo}
\begin{proof}
Let $A\in\Feh$ be such that $\sup_{A} J=c$ (which exists by Proposition \ref{prop:abstr_1}). By assumptions ($\eta$1) and ($\eta$2), $\eta(A)\in \Feh$ and $\sup_{\eta(A)} J\leq c$, hence $\sup_{\eta(A)} J= c$. Then we can find a sequence $(x_n)_n\subset A$ such that $J(\eta(x_n))\to c$. By using again assumption ($\eta2$), we infer
\[
c\geq J(x_n)\geq J(\eta(x_n))\to c,
\]
and therefore (up to a subsequence) $x_n\to\bar x\in \Kah_c$ by (PS$)_c$. On the other hand, since $A\in\Feh$, assumption ($\Feh$1) implies
that $\bar x\in A$, which concludes the proof of the theorem.
\end{proof}
%
%
Let us now turn to the asymptotic analysis. First of all we introduce a family of functionals parametrized on $\b\in(0,+\infty)$, namely $J_\b:\Mah\rightarrow \R\cup \{+\infty\}$, each of which is lower semi-continuous and moreover
\begin{itemize}
 \item[(J)] $J_{\b_1}(x)\leq J_{\b_2}(x)$ for every $x\in\Mah$, whenever $0<\b_1\leq\b_2<+\infty$.
\end{itemize}
In such a framework we define the limit functional
\[
J_\infty(x):=\sup_{\b>0} J_\b(x).
\]
\begin{lemma}\label{lem:gamma_conv}
For every $x_n,x \in \Mah$ such that $x_n\to x$ and $\b_n\to+\infty$, it holds
\[
J_\infty(x) \leq  \liminf_n J_{\b_n}(x_n).
\]
In particular, $J_\infty$ is lower semi-continuous, and $J_\b$ $\Gamma$--converges to $J_\infty$.
\end{lemma}
\begin{proof}
For every fixed $\b<+\infty$ it holds
\[
 J_\b(x)\leq\liminf_n J_\b(x_n)\leq \liminf_n J_{\b_n}(x_n)\leq \liminf_n J_\infty(x_n)
\]
(we used the fact that $J_\b$ is lower semi-continuous and that $J_\b\leq J_{\b_n}$ for $n$ sufficiently large).
Then by taking the supremum in $\b$ the lemma follows.
\end{proof}
Consequently, for $0<\b\leq+\infty$, we define the minimax levels
\[
c_\b=\inf_{A\in\Feh} \sup_{x\in A} J_\b(x).
\]
\begin{rem}\label{rem:mon_c_b}
Assumption (J) clearly yields that
\[
\b_1<\b_2<+\infty\quad\implies\quad c_{\b_1}\leq c_{\b_2}\leq c_\infty.
\]
\end{rem}
The previous remark suggests that any constant greater than $c_\infty$ is a suitable common bound for all the functionals. Hence we replace ($\Feh$2) with the following assumption.
\begin{itemize}
\item[($\Feh$2')] for any given $(A_n)_n\subset\Feh$ such that, for some $\b$, $A_n\subset \Mah_\b^{c_\infty+1}$ for every $n$, it holds $\limsup_n A_n\in \Feh$,
\end{itemize}
where
\[
\Mah_\b^{c'}=\left\{x\in\Mah:\,J_\b(x)\leq c'\right\}.
\]
Our first main result is the convergence of both the critical levels and the optimal sets (see Proposition \ref{prop:abstr_1}).
\begin{teo}\label{thm:abstr_2}
Let $J_\b:\Mah\rightarrow \R\cup \{+\infty\}$ ($0<\beta<+\infty$) be a family of lower semi-continuous functionals satisfying (J) and let $J_\infty$ be as before. Moreover suppose that assumption
($\Feh$2') holds, and that $c_\b\in \R$ for every $0<\beta\leq+\infty$. Then
\begin{enumerate}
 \item for every $0<\b<+\infty$ there exists an optimal set for $J_\b$ at $c_\b$;
 \item $c_\b\to c_\infty$ as $\b\to+\infty$;
 \item if $A_{n}\in\Feh$ is optimal for $J_{\b_n}$ at $c_{\b_n}$ and $\b_n\to+\infty$, then
 $A_\infty:=\limsup_n A_n$ is optimal for $J_{\infty}$ at $c_\infty$.
\end{enumerate}
\end{teo}
\begin{proof}
The first point is a direct consequence of Proposition \ref{prop:abstr_1}. Now by Remark \ref{rem:mon_c_b} we know that $c_\b$ is monotone in $\b$ and that $\lim c_\b \leq c_\infty<+\infty$ by assumption. Let
$\b_n, A_n, A_\infty$ be as in the statement. We have that $\sup_{A_n}J_{\b_1}\leq\sup_{A_n}J_{\b_n}\leq c_\infty$, therefore $A_n\subset \{J_{\b_1}\leq c_\infty+1\}$ and assumption ($\Feh$2') provides $A_\infty\in\Feh$. Now for every $\bar x\in A_\infty$ there exists a (sub)sequence $x_n\to \bar x$, with $x_n\in A_n$. By taking into account Lemma \ref{lem:gamma_conv} we have
\[
J_\infty(\bar x)\leq \liminf_n J_{\b_n}(x_n) \leq \liminf_n \left(\sup_{A_n} J_{\b_n}\right) = \lim_n c_{\b_n}\leq c_\infty.
\]
By taking the supremum for $\bar x\in A_\infty$ (and recalling that $A_\infty\in\Feh$), the theorem
follows.
\end{proof}

Next we turn to the study of the corresponding critical sets, by introducing a family of maps $\eta_\b: \Mah_\b^{c_{\infty}+1}\to \Mah_\b^{c_{\infty}+1}$ satisfying
\begin{itemize}
 \item[$(\eta1)_\b$] $\eta_\b(A)\in\Feh$ whenever $A\in\Feh$, $A\subset\Mah_\b^{c_{\infty}+1}$;
 \item[$(\eta2)_\b$] $J_\b(\eta_\b(x))\leq J_\b(x)$, for every $x\in\Mah_\b^{c_{\infty}+1}$.
\end{itemize}
Just as we did before, we define, for every $0<\b\leq+\infty$
\begin{equation}\label{eq:critical_sets_abstract}
\Kah_\b=\Kah_{c_\b}=\left\{x\in\Mah:\,J_\b(x)=J_\b(\eta_\b(x))=c_\b\right\}.
\end{equation}
As a straightforward consequence of Theorem \ref{thm:abstr_1}, the following holds.
\begin{teo}\label{thm:abstr_b}
Let $J_\b:\Mah\rightarrow \R\cup \{+\infty\}$ ($0<\beta<+\infty$) be a family of lower semi-continuous functionals satisfying (J) and let $J_\infty$ be as before. Suppose that $(\Feh1)$, ($\Feh$2') hold, and that, for every $0<\b\leq+\infty$, $c_\beta \in \R$ and the maps $\eta_\beta:\Mah_\b^{c_\infty+1}\rightarrow \Mah_\b^{c_\infty+1}$ verify $(\eta 1)_\b$ and $(\eta 2)_\b$. Suppose moreover that the pair $(J_\b,\eta_\b)$ satisfies $(PS)_{c_\b}$. Then every optimal set for $J_\b$ at $c_\b$ intersect $\Kah_\b$, which in particular is non empty ($\b\leq+\infty$).
\end{teo}

It is now natural to wonder what is the relation between $\limsup \Kah_\b$ and $\Kah_\infty$. The desired result would be the equality of the sets, which could be obtained under some suitable relations between the deformations $\eta_\b$ and $\eta_\infty$. However, as we mentioned in the introduction, in our application such relations do not seem to hold. Instead we will assume an uniform Palais--Smale type condition, which will lead us to consider a slightly larger set than $\limsup \Kah_{\b}$. Let us assume that the following holds:
\begin{itemize}
 \item[(UPS)] if the sequences $(x_n)_n\subset\Mah$, $(\b_n)_n\subset\R^+$ are such that $\b_n\to+\infty$ and
 $J_{\b_n} (x_n)\to c_\infty$, $J_{\b_n}(\eta_{\b_n}(x_n))\to c_\infty$, then there exists $\bar x\in \Mah$ such that, up to a subsequence,
 $x_n\to\bar x$, $\eta_{\b_n}(x_n)\to\bar x$
\end{itemize}
(again, since $J_{\b_n}(x_n)\to c_\infty$, $\eta_{\b_n}(x_n)$ is well defined for large $n$). It is worthwhile to point out explicitly the two main differences between (PS) and (UPS), apart from the
dependence on $\b$. On one hand, in the latter we do not obtain $\bar x \in \Kah_\infty$ -- see Remark \ref{rem:K_*_compact} below. On the other hand, in (UPS) we require not only $x_n$ but also $\eta_{\b_n}(x_n)$ to converge,  and the limit to be the same (to enlighten this choice, see also Remark \ref{rem:J_const=>eta_fix}). Condition (UPS) suggests the definition of the set
\[
\Cah_*=\left\{\begin{array}{l}x\in\Mah:\,\text{there exist sequences }
               (x_n)_n\subset\Mah,\,(\b_n)_n\subset\R^+\text{ such that }\\
   \qquad \bullet\ x_n\to x,\,\b_n \to +\infty,\\
   \qquad \bullet\ J_{\b_n}(x_n) \to c_\infty,\text{ and}\\
   \qquad \bullet\ J_{\b_n}(\eta_{\b_n}(x_n)) \to c_\infty
 \end{array}
 \right\}.
\]
\begin{rem}\label{rem:UPS_vs_K_*}
If $(x_n)_n$ is an uniform Palais--Smale sequence in the sense of assumption (UPS), then (up to a subsequence) $x_n\to \bar x\in\Cah_*$.
\end{rem}
\begin{rem}\label{rem:K_*_compact}
By Theorem \ref{thm:abstr_2}, it is immediate to verify that
\[
\limsup_{\b\to+\infty} \Kah_{\b}\subset\Cah_*.
\]
Finally if $x\in\Cah_*$ then, by Lemma \ref{lem:gamma_conv},
$J_\infty(x)\leq c_\infty$. Observe that the inequality may be strict (nonetheless, the following theorem will imply
that for some point the equality holds).
\end{rem}
Our final result is the following.
\begin{teo}\label{thm:abstr_3}
Under the assumptions of Theorem \ref{thm:abstr_b}, suppose moreover that (UPS) holds. Then we have that
\[
\Cah_*\cap K_\infty\neq\emptyset.
\]
More precisely, for every $(A_n)_n\subset\Feh$, with $A_n$ optimal for $J_{\b_n}$ at $c_{\b_n}$, and $\b_n\to+\infty$, there exists
$\bar x\in \Cah_*\cap K_\infty\cap \limsup_n A_n $.
\end{teo}
\begin{proof}
Let $A_n$ be as in the statement, and take $B_n=\eta_{\b_n}(A_n)$, which is also optimal for $J_{\b_n}$ at $c_{\b_n}$ by assumptions $(\eta 1)_{\b_n}$, $(\eta 2)_{\b_n}$. Theorem \ref{thm:abstr_2} then yields that $\limsup_{n} B_n=:B_\infty\in \Feh$ is optimal for $J_\infty$ at $c_\infty$, that is, there exists
$$\bar y \in B_\infty\cap K_\infty.$$
By definition, up to a subsequence, there exists $x_n\in A_n$ such that $\eta_{\b_n}(x_n)\rightarrow \bar y$. Then assumption $(\eta 2)_{\b_n}$ together with Lemma \ref{lem:gamma_conv} provides
\[
c_\infty=J_\infty(\bar y)\leq \liminf_n J_{\b_n}(\eta_{\b_n}(x_n)) \leq
\liminf_n J_{\b_n}(x_n)\leq \lim_n \left( \sup_{A_n}J_{\b_n}\right)= \lim_n c_n  = c_\infty.
\]
In particular this implies that $(x_n)_n$ is a Palais--Smale sequence in the sense of assumption (UPS); by using Remark \ref{rem:UPS_vs_K_*} we infer that (again up to a subsequence)
\[
x_n\to\bar x\in \limsup_n A_n \cap\Cah_*.
\]
But (UPS) also implies that $\eta_{\b_n}(x_n)\to \bar x$ and hence $\bar x=\bar y$, which concludes the proof of the theorem.
\end{proof}
%

\section{Convergence of the min--max levels}\label{sec:conv_lev}

The rest of the paper is devoted to apply (and refine) the results obtained in the previous section to the problem discussed in the introduction. In order to apply the abstract results of Section \ref{sec:abstrac} we need to introduce $\Mah$, $\Feh$ and
$\eta_\b$ for the present case. In this section we deal with the asymptotics of the minimax levels and prove Theorem \ref{thm:main2}. The proof of the remaining results, and in particular the construction of the deformations, will be the object of the subsequent sections.
Since the proof is independent of $k$, from now on and throughout all the paper
we assume that
\[
k\in\N^+\text{ is fixed (and will often be omitted).}
\]

We define
\[
\Mah=\left\{(u,v)\in H^1_0(\O)\times H^1_0(\O) : u,v\geq 0 \text{ in } \Omega, \,|u|_2=|v|_2=1\right\},
\]
\[
\dist^2\left((u_1,v_1),(u_2,v_2)\right)=
|u_1-u_2|_2^2+|v_1-v_2|_2^2,
\]
and
$$
J_\b(u,v)=\frac12\left(\|u\|^2+\|v\|^2\right)+\frac14\int_{\O}
\left(u^4+v^4\right)\,dx+\frac\b2\int_{\O}u^2v^2\,dx
$$
for $0<\b<+\infty$. Notice that the limiting functional (as introduced in Section \ref{sec:abstrac}) coincides with the one defined in the introduction, i.e.,
$$
J_\infty(u,v)= \sup_{\b>0} J_\b(u,v)=
\begin{cases}
J_0(u,v) & \text{when }\int_{\O}u^2v^2\,dx=0\smallskip\\
+\infty  & \text{otherwise.}
\end{cases}
$$
Moreover we set
\[
\Feh=\Feh_k=\left\{A\in\Feh_0:\,\g_2(A)\geq k\right\}\quad\text{(as in Definition \ref{defi:L2_genus})},
\]
which implies that  the critical values $c_\b$ introduced in Section \ref{sec:abstrac} coincide with the values $c_\b^k$ defined in the introduction.
\begin{rem}\label{rem:compact_level_sets}
It is worthwhile to stress that for any given $c'\in \R$ and $0<\b\leq \infty$ the set
\[
\Mah_\b^{c'}=\{(u,v)\in \Mah :\ J_\b(u,v)\leq c' \}
\]
is $L^2$--compact. This is a consequence of the coercivity of the functional together with the Sobolev embedding Theorem. This motivates our decision of working with this topology.
\end{rem}
We start by presenting some properties of the $L^2$--genus (recall Definition \ref{defi:L2_genus}).
\begin{prop}\label{prop:genus}
\begin{enumerate}
   \item[(i)] Take $A\in \mathcal{F}_0$ and let $S^{k-1}$ be the standard ($k-1$)-sphere in $\R^k$.
   If there exists an
   $L^2$--homeomorphism $\psi: S^{k-1} \to A$ satisfying $\psi(-x)=\sigma (\psi(x))$ then $\g_2(A)=k$.
   \item[(ii)] Consider $A\in \mathcal{F}_k$ and let $\eta:A\rightarrow \Mah$ be an $L^2$--continuous,
   $\sigma$--equivariant and sign--preserving map. Then $\overline{\eta(A)}\in \mathcal{F}_k$.
   \item[(iii)] If $A\in \Feh_0$ is an $L^2$--compact set, then there exists a $\delta>0$ such
   that\footnote{Here $N_\delta(A)=\left\{(u,v)\in\Mah:\dist_2((u,v),A)<\delta\right\}$.}
   $\g_2\left(\overline{N_\delta(A)}\right)=\g_2(A)$.
   \item[(iv)] Let $\{A_n\}_{n\in \N}$ be a sequence in $\mathcal{F}_k$ and let $X$ be an $L^2$--compact subset of
   $\Mah$  such that $A_n\subset X$. Then \footnote{Recall that $\limsup A_n=\{(u,v)\in \Mah:
   \exists n_k\rightarrow +\infty,\ (u_{n_k},v_{n_k})\in A_{n_k} \text{ such that }
   \dist_2((u_{n_k},v_{n_k}), (u,v))\rightarrow 0\}$.} $\limsup A_n\in\mathcal{F}_k$.
\end{enumerate}
\end{prop}
\begin{proof}
The proofs of the first three properties are similar to the ones of the usual genus, and therefore we omit them (see for example Struwe, Proposition 5.4). As for (iv), let $A_n$ and $X$ be as above. By the definition of
$\limsup$ it is straightforward to check that the set $\limsup_n A_n$ belongs to $\Feh_0$, and that it is $L^2$--compact. We now claim that for every $\delta>0$ there exists $n_0\in\N$ such that
$$
A_n\subset N_\delta(\limsup A_n) \qquad \text{for } n\geq n_0,
$$
which together with point (iii) yields the desired result. Suppose that our claim is false. Then there exist a $\bar \delta>0$,
$n_k\rightarrow +\infty$ and $(u_{n_k},v_{n_k})\in A_{n_k}$ such that
$(u_{n_k},v_{n_k})\notin N_{\bar \delta}(\limsup A_n).$ But since $X$ is
sequentially compact, then there exists a $(u,v)\in X\subset \Mah$ such that, up
to a subsequence, $(u_{n_k},v_{n_k})\rightarrow (u,v)$. Hence
$(u,v)\in \limsup A_n$, a contradiction.
\end{proof}

\begin{lemma}\label{lem:G_k}
For every $\b$ finite it holds $0\leq c_\b\leq c_\infty<+\infty$.
\end{lemma}
\begin{proof}
The proof of the lemma relies on the fact that, given any
$k\in \N$, we can construct a set $G_k\in\mathcal{F}_0$ with $\g_2(G_k)=k$. Here we use some ideas presented in \cite{dww}, Proposition 4.3. Indeed, consider $k$ functions $\phi_1, \ldots,\phi_k \in
H^1_0(\Omega)$ such that $\phi_i\cdot \phi_j =0$ a.e. for any $i\neq
j$, with $\phi_i^+,\phi_i^-\neq 0.$ Define
\[
\psi: S^{k-1} \rightarrow \Mah,\qquad (t_1,\ldots,t_k)\mapsto \left(\bar{t} \left(\sum_i t_i
\phi_i\right)^+,\bar{s} \left(\sum_i t_i \phi_i\right)^-\right),
\]
where
\[
\bar{t}^2=\frac{1}{\left|\left(\sum_i t_i \phi_i\right)^+\right|^2_2}= \frac{1}{\left(\sum_i t_i^2
\left|\phi_i^+\right|^2_2\right)}, \qquad \bar{s}^2=\frac{1}{\left|\left(\sum_i t_i \phi_i\right)^-\right|^2_2}= \frac{1}{\left(\sum_i t_i^2
\left|\phi_i^-\right|^2_2\right)},
\]
and $G_k=\psi(S^{k-1})$. It is easy to verify that $G_k\in\Feh_0$. Since $\psi$ is an $L^2$--homeomorphism between
$S^{k-1}$ and $G_k$, and
\(
\s(\psi(t_1,\ldots,t_k))=\psi(-t_1,\ldots,-t_k),
\)
then Proposition \ref{prop:genus}-$(i)$ provides that $\g_2(G_k)=k$. Since $(u,v)\in G_k$ implies $u\cdot v\equiv0$, then
\[
c_\infty\leq\sup_{G_k}J_\infty <+\infty.
\]
Finally, Remark \ref{rem:mon_c_b} allows to conclude the proof.
\end{proof}
We are already in a position to prove the convergence of the minimax levels.

\begin{proof}[Proof of Theorem \ref{thm:main2}]
This is a direct consequence of Theorem \ref{thm:abstr_2}. Let us check its hypotheses. Under the above definitions, assumption (J) easily holds. For every $0<\b\leq +\infty$, $c_\b \in \R$ (by Lemma \ref{lem:G_k}), and moreover
($\Feh$2') holds (by recalling Remark \ref{rem:compact_level_sets}, Proposition \ref{prop:genus}-(iv) and by using the fact that $c_\infty\in \R)$. Finally let us check that each $J_\b$ is a lower semi-continuous functional in $(\Mah,\dist)$, for $0<\beta<+\infty$.  Indeed,
let $(u_n,v_n)$, $(\bar u,\bar v)$ be couples of $H^1_0$ functions
such that dist$((u_n,v_n),(\bar u,\bar v))\to0$. If $\liminf_n
J_\b(u_n,v_n)=+\infty$ then there is nothing to prove, otherwise,
by passing to the subsequence that achieves the $\liminf$, we have that
$\|(u_n,v_n)\|$ is bounded. Thus, again
up to a subsequence, $(u_n,v_n)$ weakly converges (in $H^1_0$), and, by
uniqueness, the weak limit is $(\bar u,\bar v)$. Then we can
conclude by using the weak lower semicontinuity of $\|\cdot\|$ (and the
weak continuity of the other terms in $J_\b$).
\end{proof}
Let us conclude this section recalling that, if $\b$ is sufficiently large,
we can exclude the presence of fixed points of $\sigma$ in the set $\mathcal{K}^k_\b$. As in the usual
genus theory, this insures that, if two (or more) critical values coincide, then $\mathcal{K}^k_\b$ contains an infinite number of elements.
\begin{lemma}\label{lem:no_fix_point}
Let $k\in\N$ be fixed. There exists a (finite) number $\bar\b(k)>0$,
depending only on $k$, such that, for every
$\bar\b(k)\leq\b\leq+\infty$, we have
\[
\mathcal{K}^k_\b \cap\left\{(u,u)\in \Mah\right\}=\emptyset.
\]
\end{lemma}

\begin{proof}
When $\b=+\infty$ the assertion holds true with no limitations on $\b$, since $J_\infty(u,u)<+\infty$ implies $u\equiv0$, and $(0,0)\not\in \Mah$. For $\b<+\infty$ let us consider the problem
\[
\inf_{(u,u)\in \Mah} J_\b(u,v) = \inf_{|u|_2=1} \left(\|u\|^2 +\frac{1+\b}{2}\int_\O u^4\,dx\right)\geq\inf_{|u|_2=1} \frac{1+\b}{2|\O|}
\left(\int_\O u^2\,dx\right)^2= \frac{1+\b}{2|\O|}.
\]
Taking into account Lemma \ref{lem:G_k}, the assertion of the lemma is proved once
\[
\frac{1+\b}{2|\O|}>c_k^\infty.
\]
But this is true if we take $\b\geq\bar\b(k)=2|\O|c_k^\infty$.
\end{proof}

\section{Existence and asymptotics of the critical points}\label{sec:main1_beta}

In this section we prove the remaining results stated in the introduction. To this aim we shall define suitable deformations $\eta_\b$, which will allow us to apply the abstract results of Section \ref{sec:abstrac} that concern the critical sets -- namely Theorems \ref{thm:abstr_b} and \ref{thm:abstr_3}. Afterwards, we will establish the equivalence between the critical sets defined in the introduction and the ones of Section \ref{sec:abstrac}.

As we mentioned, we need to choose different deformations for our porpoises, for the case $\b<+\infty$ and
$\b=+\infty$. Let us start with the definition of $\eta_\b$ for $\b<+\infty$ (here $\b$ is fixed). The desired map will make use of the parabolic flow associated to $J_\b$ on $\Mah$. In order to do so, first we need to fix a relation between $(\l,\mu)$ and $(u,v)$.
\begin{rem}\label{rem:lagr_mult}
If $(u,v)\in \Mah$ satisfies \eqref{eq:sys} then, by testing the equations
with $u$ and $v$ respectively, one immediately obtains
\[
\l=\l(u,v)=\frac{\int_\O \left(|\nabla u|^2+u^4+\b u^2v^2\right)\,dx}{\int_\O u^2\,dx}=
\int_\O \left(|\nabla u|^2+u^4+\b u^2v^2\right)\,dx,
\]
\[
\mu=\mu(u,v)=\frac{\int_\O \left(|\nabla v|^2+v^4+\b u^2v^2\right)\,dx}{\int_\O v^2\,dx}=
\int_\O \left(|\nabla v|^2+v^4+\b u^2v^2\right)\,dx.
\]
\end{rem}
Motivated by the previous remark and by the definition of $S_\b$  (see \eqref{eq:sys}), we write, with some abuse of notations,
\begin{equation}\label{eq:esse_beta}
S_\b(u,v)=S_\b(u,v;\l(u,v),\mu(u,v)),
\end{equation}
with $\l$, $\mu$ as above.
Then, for $(u,v)\in \Mah$, we consider the initial
value problem with unknowns $U(x,t)$, $V(x,t)$,
\begin{equation}\label{eq:flux}
\left\{
\begin{array}{l}
 \partial_t(U,V)=-S_\b(U,V)\smallskip\\
 U(\cdot,t),V(\cdot,t)\in\spc\\
 U(x,0)=u(x),\quad V(x,0)=v(x),
\end{array}
\right.
\end{equation}
We have the following existence result.
\begin{lemma}\label{lem:weissler_exist}
For every $(u,v)\in \Mah_\b^{c_\infty+1}$ problem \eqref{eq:flux} has exactly one solution
\[
(U(t),V(t))\in C^1\left((0,+\infty); L^2(\O)\times
L^2(\O)\right)\cap C\left([0,+\infty); \spc\times\spc\right).
\]
Moreover, for every $t>0$, $|(U(t),V(t))|_2=1$ and
\[
\frac{d}{dt} J_\b(U(t),V(t))= -\left|S_\b(U(t),V(t))\right|^2_2\leq0.
\]
\end{lemma}

We postpone to Section \ref{sec:construction_fluxes} the proofs of this result and of the subsequent properties.

\begin{prop}\label{prop:flux_beta_finite}
Using the notations of Lemma \ref{lem:weissler_exist}, the following properties hold
\begin{itemize}
\item[(i)] $U(t)\geq0, V(t)\geq0$, for every $(u,v)\in \Mah_\b^{c_\infty+1}$ and $t>0$;
\item[(ii)] for every fixed $t>0$ the map $(u,v)\mapsto(U(t),V(t))$ is $L^2$--continuous from $\Mah_\b^{c_\infty+1}$ into itself;
\item[(iii)] let $(u,v)\in \Mah_\b^{c_\infty+1}$, $s,t \in[0,+\infty)$ then
\[
\dist\left((U(s),V(s)),(U(t),V(t))\right)\leq |t-s|^{1/2} |J_\b(U(s),V(s))-J_\b(U(t),V(t))|^{1/2}.
\]
\end{itemize}
\end{prop}

All the previous results allow us to define an appropriate deformation, along with some key properties.
\begin{prop}\label{prop:deformation_beta_finite}
Let us define, under the above notations,
\[
\eta_\b:\Mah_\b^{c_\infty+1}\to \Mah_\b^{c_\infty+1}, \qquad (u,v) \mapsto \eta_\b(u,v)=(U(1),V(1)).
\]
Then $\eta_\b$ satisfies assumptions $(\eta 1)_\b$ and $(\eta 2)_\b$.
\end{prop}
\begin{proof}
Lemma \ref{lem:weissler_exist} implies that
$$J_\b(\eta_\b(u,v))=J_\b(U(1),V(1))\leq J_\b(U(0),V(0))=J_\b(u,v)$$for every $(u,v)$, which is exactly assumption $(\eta 2)_\b$. This together with Proposition \ref{prop:flux_beta_finite}-(i) also implies that, as stated, $\eta_\b\left( \Mah_\b^{c_\infty+1}\right)\subseteq \Mah_\b^{c_\infty+1}$. Moreover we observe that $\eta_\b$ is $\sigma$-- equivariant (by the uniqueness of the initial value problem \eqref{eq:flux}) and that it is $L^2$--continuous (Proposition \ref{prop:flux_beta_finite}-(ii)). Thus Proposition \ref{prop:genus}-(ii) applies, yielding $\overline{\eta_\b(A)}\in \Feh_k$. Since $A$ is $L^2$--compact in $\Mah$ (indeed it is a closed subset of the
$L^2$--compact $\Mah_\beta^{c_\infty}$) then $\eta_\b(A)$ is closed, and therefore assumption $(\eta 1)_\b$ holds.
\end{proof}
Before moving to the infinite case, let us prove the validity of a Palais--Smale type condition. It will be the key ingredient in order to show that $(J_\b,\eta_\b)$ satisfies (PS$)_{c_\b}$ according to Definition \ref{def:palais_smale_abstrac}.
\begin{lemma}\label{lem:ps1}
Let $(u_n,v_n)\in \Mah$ be such that, as $n\to +\infty$,
\[
J_\b(u_n,v_n) \to c \qquad\text{and}\qquad |S_\b(u_n,v_n)|_2 \to 0
\]
for some $c\geq 0$.
Then there exists $(\bar u,\bar v)\in \Mah\cap (H^2(\Omega)\times H^2(\Omega))$ such that, up to a subsequence,
\[
(u_n,v_n)\to(\bar u,\bar v)\text{ strongly in }H^1_0\qquad\text{and}\qquad S_\b(\bar u,\bar v)=0.
\]
\end{lemma}
\begin{proof}
Since $J_\b(u_n,v_n)\rightarrow c$, then we immediately infer the existence of $(\bar u,\bar v) \in \Mah$ such that $(u_n,v_n)\rightharpoonup (\bar u,\bar v)$ weakly in $H^1_0$, up to a subsequence. Let us first prove the $H^1_0$--strong convergence. From the fact that $|S_\b(u_n,v_n)|_2\to0$ and that $u_n-\bar u$ is $L^2$--bounded, we deduce
$$\langle S_\b(u_n,v_n),(u_n-\bar u,0)\rangle_2 = \int_\Omega [ \nabla u_n \cdot \nabla(u_n-\bar u)+(u_n^3+\b u_nv_n^2-\lambda(u_n,v_n)u_n)(u_n-\bar u)]dx \rightarrow 0.$$
This, together with
\[
\begin{split}
\left|\int_\Omega (u_n^3+\b u_nv_n^2-\lambda(u_n,v_n)u_n)(u_n-\bar u)dx\right| & \leq |u_n^3+\b u_n v_n^2-\lambda(u_n,v_n)u_n|_2 |u_n-\bar u|_2\\
&\leq C |u_n-\bar u|_2\rightarrow 0,
\end{split}
\]
implies that $\int_\Omega \nabla u_n \cdot \nabla (u_n-\bar u) \to 0$, yielding the desired convergence. The fact that $v_n\to \bar v$ can be proved in a similar way.

Now we pass to the proof of the last part of the statement. A first observations is that
\[
|\Delta u_n|_2^2+|\Delta v_n|_2^2\leq 2|S_\b(u_n,v_n)|_2^2+2|u_n^3 + \b u_n v_n^2 -\lambda(u_n,v_n)u_n|_2^2+2|v_n^3+\b u_n^2 v_n -\mu(u_n,v_n)v_n|_2^2\leq C,
\]
which yields the weak $H^2$--convergence $u_n\rightharpoonup \bar u, v_n\rightharpoonup \bar v$ (up to a subsequence). As a consequence, we have that $\langle S_\b(u_n,v_n),(\phi,\psi)\rangle_2\rightarrow \langle S_\b(\bar u,\bar v),(\phi,\psi)\rangle_2$  for any given $(\phi,\psi)\in L^2$.
On the other hand, $|S_\b(u_n,v_n)|_2\to 0$ provides that
$$
\langle S_\b(u_n,v_n),(\phi,\psi)\rangle_2\to 0,
$$
thus $S_\b(\bar u,\bar v)=0$ and the lemma is proved.
\end{proof}

Let us turn to the definition of the deformation $\eta_\infty$. The main difficulty in this direction is that $J_\infty$ is finite if and only if $uv\equiv0$, thus any flux we wish to use must preserve the disjointness of the supports. As we said in the introduction, here the criticality condition will be given by equation \eqref{eq:eq}. In order to overcome the lack of regularity due to the presence of the positive/negative parts in the equation, we will use a suitable gradient flow, instead of a parabolic flow. More precisely we define
\[
\tS: \spc  \to  H^1_0(\O)
\]
to be the gradient of the functional $J^*(w)$ (see equation \eqref{eq:J*}), constrained to the set $\int_\O(w^+)^2=\int_\O(w^-)^2=1$. If $\invD$ denotes the inverse of $-\D$ with Dirichlet boundary conditions, then we will prove in Section \ref{sec:construction_fluxes} the following result.
\begin{lemma}\label{lem:lambda_mu_tilde}
Let $R_1,R_2>0$ be fixed. For every $w\in \spc$ such that
\[
|w^+|_2,|w^-|_2\geq R_1\qquad\text{and}\qquad \|w\|\leq R_2
\]
there exist unique $\tl=\tl(w)$, $\tmu=\tmu(w)$ such that
\[
\tS(w) = w + \invD \left(w^3 -\tl w^+ +\tmu w^-\right).
\]
Moreover, $\tl$ and $\tmu$ are Lipschitz continuous in $w$ with respect to the $L^2$--topology, with Lipschitz constants only depending on $R_1,R_2$.
\end{lemma}
For every $(u,v)\in \Mah_\infty^{c_\infty+1}$ we consider the initial value problem (with unknown $W=W(t,x)$)
\begin{equation}\label{eq:flux_infty}
\left\{
\begin{array}{l}
 \partial_t W=-\tS(W)\smallskip\\
 W(\cdot,t)\in\spc\\
 W(x,0)=u(x)-v(x).
\end{array}
\right.
\end{equation}
and prove existence and regularity of the solution.
\begin{lemma}\label{lem:flux_infty_exist}
For every $(u,v)\in \Mah_\infty^{c_\infty+1}$ problem \eqref{eq:flux_infty} has exactly one solution
\[
W(t)\in C^1\left((0,+\infty); \spc\right)\cap C\left([0,+\infty); \spc\right).
\]
Moreover, for every $t$, $(W^+(t),W^-(t))\in \Mah_\infty^{c_\infty+1}$ and
\[
\frac{d}{dt} J_\infty(W^+(t),W^-(t))= -\|\tS(W(t))\|^2\leq0.
\]
\end{lemma}

Again, the proof of this result can be found in Section \ref{sec:construction_fluxes}, together with the proof of the following properties.

\begin{prop}\label{prop:flux_beta_infinite}
Using the notations of Lemma \ref{lem:flux_infty_exist}, the following properties hold
\begin{itemize}
\item[(i)] for every fixed $t>0$ the map $(u,v)\mapsto(W^+(t),W^-(t))$ is $L^2$--continuous from
$\Mah_\infty^{c_\infty+1}$ into itself;
\item[(ii)] let $(u,v)\in \Mah_\infty^{c_\infty+1}$, $s,t \in[0,+\infty)$ then\footnote{Here $C_S$ is the Sobolev constant of the embedding $H^1_0\hookrightarrow L^2$.}
\[
\dist((W^+(s),W^-(s)),(W^+(t),W^-(t))) \leq C_S |t-s|^{1/2} |J_\infty(W^+(s),W^-(s))-J_\infty(W^+(t),W^-(t))|^{1/2}.
\]
\end{itemize}
\end{prop}

Similarly to the case $\b$ finite, the previous properties allow to define a suitable deformation (we omit the proof since it is similar to the case $\b$ finite).
\begin{prop}
Let us define, under the above notations,
\[
\eta_\infty: \Mah_\infty^{c_\infty+1}\to \Mah_\infty^{c_\infty+1},\qquad (u,v)\mapsto\eta_\infty(u,v)=(W^+(1),W^-(1)).
\]
Then $\eta_\infty$ satisfies assumptions $(\eta1)_\infty$ and $(\eta2)_\infty$.
\end{prop}

Turning to the Palais--Smale condition, here is a preliminary result.

\begin{lemma}\label{lem:PS_infty}
Let $(u_n,v_n)\in \Mah_\infty^{c_\infty+1}$ be such that, as $n\to +\infty$,
\[
J_\infty(u_n,v_n)\to c_\infty \qquad\text{and}\qquad \|\tS(u_n-v_n)\|\to 0.
\]
Then there exists $\bar w\in \spc$ such that, up to a subsequence,
\[
u_n-v_n\to\bar w\text{ strongly in }H^1_0\qquad\text{and}\qquad \tS(\bar w)=0.
\]
\end{lemma}
\begin{proof}
Let $(w_1,w_2)$ be such that, up to subsequences, $u_n\rightharpoonup w_1, v_n\rightharpoonup w_2$ in $H^1_0(\O)$. Since $J_\infty(u_n,v_n)<\infty$, then $u_n\cdot v_n=0$ and therefore also $w_1\cdot w_2=0$. Denote $w_n=u_n-v_n$ and $\bar w=w_1-w_2$ in such a way that
\[
\tS(u_n-v_n)=w_n + (-\Delta)^{-1}(w_n^3-\tl(w_n)w_n^++\tmu(w_n)w_n^-).
\]
Let us prove the $H^1_0$--convergence. First observe that $w_n-\bar w$ is bounded in $H^1_0$, which implies that $-\Delta(w_n-\bar w)$ is $H^{-1}$--bounded. Now since $\|\tS(u_n-v_n)\|\to 0$ we obtain
\begin{multline*}
\langle -\Delta(w_n-\bar w),\tS(u_n-v_n)\rangle_{H^{-1}}= \int_\Omega (\nabla w_n\cdot \nabla (w_n-\bar w) + w_n^3(w_n-\bar w)-\\
  -\tl(w_n)w_n^+(w_n-\bar w)+\tmu(w_n)w_n^-(w_n-\bar w))\,dx\to 0.
\end{multline*}
This, together with the fact that
\begin{multline*}
\left|\int_\Omega (w_n^3(w_n-\bar w)-\tl(w_n)w_n^+(w_n-\bar w)+\tmu(w_n)w_n^-(w_n-\bar w))\,dx\right| \leq \\
\leq |w_n^3-\tl(w_n)w_n^+ +\tmu(w_n)w_n^-|_2 |w_n-\bar w|_2 \to 0
\end{multline*}
gives $\displaystyle \int_\Omega \nabla w_n \cdot \nabla (w_n-\bar w)\to 0$, which yields the $H^1_0$--convergence of $w_n$ to $\bar w$.

In order to conclude the proof of the lemma it remains to show that $\tS(\bar w)=0$. Now, $w_n\to \bar w$ in $H^1_0$ implies that $w_n^3-\tl(w_n)w_n^+-\tmu(w_n)w_n^-$ is bounded in $L^2$ which, together with the fact that $(-\Delta)^{-1}$ is a compact operator from $L^2(\Omega)$ to $H^1_0(\Omega)$ provides, up to a subsequence, the convergence
\[
 (-\Delta)^{-1}(w_n^3-\tl(w_n)w_n^+-\tmu(w_n)w_n^-)\to (-\Delta)^{-1}(\bar w^3-\tl(\bar w)\bar w^+ + \tmu(\bar w)\bar w^-) \text{ in } H^1_0(\Omega).
\]
Hence also $\tS(u_n-v_n)\to \tS(\bar w)$ in $H^1_0(\Omega)$, which concludes the proof.
\end{proof}

We are ready to show that the deformations we have defined satisfy the remaining abstract properties required in Section \ref{sec:abstrac}.

\begin{prop}\label{prop:palais_smale_beta_finite}
For every $0<\b\leq+\infty$, the pair $(J_\b, \eta_\b)$ satisfies (PS$)_{c_\b}$ (according to Definition \ref{def:palais_smale_abstrac}).
\end{prop}
\begin{proof}
Let first $\b<\infty$ fixed. Let $(u_n,v_n)\subset \Mah$ be a Palais--Smale sequence in the sense of Definition \ref{def:palais_smale_abstrac}, that is, $J_\b(u_n,v_n)\to c_\b$ and $J_\b(\eta_\b(u_n,v_n))\to c_\b$. Let then $(\bar u, \bar v)\in \Mah$ be such that, up to a subsequence, $(u_n,v_n) \to (\bar u,\bar v)$ in $L^2$. Define $(U_n(t), V_n(t))$ as the solution of (\ref{eq:flux}) with initial datum $(u_n,v_n)$ (recall that therefore $\eta_\b(u_n,v_n)=(U_n(1),V_n(1))$). By applying Proposition \ref{prop:flux_beta_finite}-(iii) with $(s,t)=(0,1)$ we obtain
\[
\dist((u_n,v_n),\eta_\beta(u_n,v_n)) \leq  |J_\b(u_n,v_n)-J_\b(\eta_\b(u_n,v_n))|^{1/2}\to 0,
\]
which, together with the $L^2$--continuity of $\eta_\b$, yields $(\bar u, \bar v)=\eta_\b(\bar u, \bar v)$. It only remains to show that $J_\b(\bar u, \bar v)=c_\b$.
Notice that

\[
\int_0^1 |S_\b(U_n(t),V_n(t))|_2^2 dt = J_\b(u_n,v_n)-J_\b(\eta_\b(u_n,v_n)) \to 0,
\]
(by Lemma \ref{lem:weissler_exist}) and hence, for almost every $t$, $|S_\b(U_n(t), V_n(t))|_2\to 0$
(up to a subsequence). Moreover, being $J_\b$ a decreasing functional under the heat flux, it holds $J_\b(U_n(t), V_n(t)) \to c_\b$. Now Lemma \ref{lem:ps1} applies providing the existence of $(u,v)\in \Mah$ such that $(U_n(t), V_n(t))\to (u,v)$ in $H^1_0$, and in particular $J_\b(u,v)=c_\b$. Finally the use of Proposition \ref{prop:flux_beta_finite}-(iii) with $(s,t)=(0,t)$ allows us to conclude that $(u,v)=(\bar u, \bar v)$, and the proof is completed.

The case $\b=+\infty$ can be treated similarly, substituting $(U(t),V(t))$ with $(W^+(t), W^-(t))$ and $|S_\b|_2$ with $\|\tS\|$.
\end{proof}
An uniform Palais--Smale condition also holds, in the sense of assumption (UPS). The proof of this fact is very similar to the one of Proposition \ref{prop:palais_smale_beta_finite}, and hence we omit it.
\begin{prop}
Assumption (UPS) holds.
\end{prop}
The properties collected in this section show that Theorems \ref{thm:abstr_b} and \ref{thm:abstr_3} apply to this framework. Thus we are in a position to conclude the proofs of the results stated in the introduction.
\begin{proof}[End of the proof of Theorem \ref{thm:main1}] As Theorem \ref{thm:abstr_b} holds, the last thing we have to check is that the critical set $\Kah_\b$ (according to \eqref{eq:critical_sets_abstract}) coincides with the one defined in the introduction. Again, we only present a proof in the case $\b<+\infty$. We have to show that $J_\b(u,v)=J_\b(U(1),V(1))$ if and only if $S_\b(u,v)=0$. But this readily follows from the fact that, for $t\in[0,1]$,
\[
\dist((u,v),(U(t),V(t)))^2\leq\int_0^1\left|S_\b(U(\tau),V(\tau))\right|_2^2d\tau= J_\b(u,v)-J_\b(U(1),V(1)),
\]
once one observes that, by uniqueness, $(U(t),V(t))\equiv (u,v)$ if and only if $S_\b(u,v)=0$. Finally, the
$H^1$--compactness of $\Kah_\b$ comes directly from Lemmas \ref{lem:ps1} and \ref{lem:PS_infty}.
\end{proof}
%

%
%
\begin{proof}[End of the proof of Theorem \ref{thm:main3}] As Theorem \ref{thm:abstr_3} holds, the result
is proved once we show that $\Cah_*\subset \Kah_*$. To this aim, let us consider $(u,v)\in\Cah_*$ and let,
by definition, $(u_n,v_n)\in\Mah$ be such that $(u_n,v_n)\to(u,v)$ in $L^2$, $J_{\b_n}(u_n,v_n)\to c_\infty$ and $J_{\b_n}(U_n(1),V_n(1))\to c_\infty$. By arguing exactly as in the proof of Proposition \ref{prop:palais_smale_beta_finite}, we infer the existence of a $0\leq t\leq1$ such that it holds
$(U_n(t),V_n(t))\to(u,v)$, $J_{\b_n}(U_n(t),V_n(t))\to c_\infty$ and $|S_{\b_n}(U_n(t),V_n(t))|_2\to 0$.
Therefore $(u,v)\in\Kah_*$.
\end{proof}

\begin{proof}[Proof of Corollary \ref{coro:intro_last}]
The only thing left to prove is that,
given any $(u_n,v_n)\in \Mah$ and $\b_n\to+\infty$ such that $(u_n,v_n)\to(\bar u,\bar v)$ in $L^2$, with
\(
J_{\b_n}(u_n,v_n) \to c_\infty,\ |S_{\b_n}(u_n,v_n)|_2 \to 0
\), then in fact $(u_n,v_n)\to(\bar u,\bar v)$ in $H^1\cap C^{0,\alpha}$.
We shall prove that the sequence $(u_n,v_n)$ is uniformly bounded in the $L^\infty$--norm. This, together with the fact that, by assumption,
\[
   \begin{array}{l}
   -\D u_\b + u_\b^3+\b u_\b v_\b^2-\l_\b u_\b=h_\b\to0\quad\text{in }L^2\smallskip\\
   -\D v_\b + v_\b^3+\b u_\b^2v_\b-\mu_\b v_\b=k_\b\to0\quad\text{in }L^2,\\
 \end{array}
\]
allows us to apply Theorem 1.4 in \cite{nttv1}, which provides the desired result.

Since $J_{\b_n}(u_n,v_n)\rightarrow c_\infty$, we infer the existence of $\l_{\max},\mu_{\max} \in \R$ such that, up to a subsequence,
\[
(u_n,v_n)\rightharpoonup (\bar u,\bar v) \text{ in } H^1_0, \qquad \l(u_n,v_n)\leq \l_{\max}, \ \mu(u_n,v_n)\leq \mu_{\max}, \ \ \forall n.
\]
In order to prove uniform bounds in the $L^\infty$--norm, we shall apply a Brezis--Kato type argument to the sequence $(u_n,v_n)$. Suppose $u_n \in L^{2+2\d}(\O)$ for some $\d>0$; we can test with $u_n^{1+\d}$ the inequality
\[
-\Delta u_n \leq \l(u_n,v_n) u_n +h_n,
\]
obtaining
\[
\frac{1+\d}{\left(1+\frac{\d}{2}\right)^2}\int_\O |\nabla(u_n^{1+\frac{\d}{2}})|^2\,dx\leq
\l(u_n,v_n)\int_\O u_n^{2+\d}\,dx+\int_\O h_n u_n^{1+\d}\,dx.
\]
Hence, by Sobolev embedding we have\footnote{Here $C_S$ denotes the Sobolev constant of the embedding $H^1_0\hookrightarrow L^6$.}
\[
|u_n|_{6+3\d}\leq \left[ C_S^2 \frac{\left(1+\frac{\d}{2}\right)^2}{1+\d} \right]^{\frac{1}{2+\d}}
\left[\l(u_n,v_n)\int_\O u_n^{2+\d}\,dx+\int_\O h_n u_n^{1+\d}\,dx\right]^{\frac{1}{2+\d}}.
\]
Now apply H\"older inequality to the right hand side; provided $\int_\O u_n^{2+2\d}\,dx\geq 1$, there holds
\[
\l(u_n,v_n)\int_\O u_n^{2+\d}\,dx \leq \l_{\text max} |\O|^{1/2} |u_n|_{2+2\d}^{2+\d} \qquad\text{and}\qquad
\int_\O h_n u_n^{1+\d}\,dx \leq |h_n|_2 |u_n|_{2+2\d}^{2+\d}
\]
hence, since $|h_n|_2\to 0$, we have proved the existence of a constant $C$, not depending on $n$ and $\d$ such that
\[
|u_n|_{6+3\d}\leq \left[ C_S^2 \frac{\left(1+\frac{\d}{2}\right)^2}{1+\d} \right]^{\frac{1}{2+\d}} |u_n|_{2+2\d}.
\]
Now iterate, letting
\[
\d(1)=2, \ 2+2\d(k+1)=6+3\d(k) \qquad\text{hence}\qquad \d(k)\geq\left(\frac{3}{2}\right)^{k-1}.
\]
If there exist infinite values $\delta(k)$ such that $\int_\O u_n^{2+2\d(k)}\,dx<1$, the $L^\infty$--estimate is trivially proved; otherwise the previous estimates hold for $\d(k)$ sufficiently large providing, for every $p>1$,
\[
|u_n|_p \leq C' + \prod_{k=1}^{+\infty}\left[ C \frac{\left(1+\frac{\d(k)}{2}\right)^2}{1+\d(k)} \right]^{\frac{1}{2+\d(k)}} |u_n|_6.
\]
The last inequality provides the desired $L^\infty$ estimate since it is easy to verify that
\[
\sum_{k=1}^\infty \frac{1}{2+\d(k)}\log\left[ C \frac{\left(1+\frac{\d(k)}{2}\right)^2}{1+\d(k)} \right] < \infty, \qquad\text{if}\qquad
\d(k)\geq\left(\frac{3}{2}\right)^{k-1}.
\]
The same calculations clearly hold for $v_n$.
\end{proof}
We conclude by giving a proof of Theorem \ref{teo:ground_states} as a particular case of the theory we developed (although, as we mentioned, it is possible to give a more elementary proof of this result).
\begin{proof}[Proof of Theorem \ref{teo:ground_states}]
The key remark in this framework is that, in fact, for every $0<\b\leq+\infty$ we can write
\[
c^1_\b=\inf_{(u,v)\in \Mah} J_\b(u,v).
\]
More precisely,
\[
(u_\b,v_\b)\text{ achieves }c^1_\b\quad\implies\quad A_\b=\{(u_\b,v_\b),(v_\b,u_\b)\}\text{ is an optimal set
for }J_\b\text{ at }c^1_\b.
\]
Now, the $L^2$--convergence of the minima follows by the convergence of the optimal sets (Theorem \ref{thm:main2}),
while the $H^1\cap C^{0,\alpha}$--convergence is obtained as in the previous proof.
\end{proof}

\section{Construction of the flows}\label{sec:construction_fluxes}

\begin{proof}[Proof of Lemma \ref{lem:weissler_exist}]
In order to prove local existence, we want to apply Theorem 2,
b) in \cite{weissler}, to which we refer for further details.
Let us rewrite the problem as
\[
w'=\Delta w + F(w),
\]
where $w=(U,V)$, $w'=\partial_t(U,V)$, $\Delta$ is intended in the vectorial sense and
$F$ contains all the remaining terms. Using the notations
of \cite{weissler} we set $E=L^2(\O)\times L^2(\O)$ and
$E_F=\spc\times\spc$. We obtain that $e^{t\Delta}$ is an
analytic semigroup both on $E$ and on $E_F$,
satisfying\footnote{By using for example the expansion in
eigenfunctions of $-\Delta$ in $H^1_0$, one can easily obtain the
required inequality with $C=(2e)^{-1/2}$.}
\[
\|e^{t\Delta}w_0\|\leq C t^{-1/2}|w_0|_2\qquad\text{for every } w_0\in E,
\]
so that $(2.1)$ in \cite{weissler} holds with $a=1/2$. Moreover,
since all the terms in $F$ are of polynomial type, it is
easy to see that $F:E_F\to E$ is locally lipschitz
continuous, and
\[
|F(w_0)-F(z_0)|_2\leq\ell(r)\|w_0-z_0\|,\text{ with }
\ell(r)=O(r^p)\text{ as }r\to+\infty,
\]
whenever $\|w_0\|\leq r,\,\|z_0\|\leq r$ (for example, arguing as in Lemma \ref{lem:cont_L2_2}, the previous
estimate holds for $p=4$). Now, choosing $b=1/(2p)<a$, it is
immediate to check that
\[
\ell(r)=O\left(r^{(1-a)/b}\right),
\]
thus $(2.3)$ in \cite{weissler} is also satisfied. In order to apply
Theorem 2, b)  the last assumption we need to verify is that, for every $w_0\in H^1_0$ (which is our regularity assumption for the initial data in \eqref{eq:flux}), it holds
\[
\limsup_{t\downarrow0}\|t^be^{t\Delta}w_0\|=0;
\]
but this follows recalling that $\|e^{t\Delta}w_0\|\leq \|w_0\|$\footnote{Again, one can obtain this
inequality expanding in eigenfunctions.}. Therefore Theorem 2, b) and Corollary 2.1, b) and c) in \cite{weissler} apply, providing the existence of a (unique) maximal solution of \eqref{eq:flux}
\[
(U(t),V(t))\in C^1\left((0,T_{\max}); L^2(\O)\times
L^2(\O)\right)\cap C\left([0,T_{\max}); \spc\times\spc\right),
\]
with the property that if $T_{\max}<+\infty$ then $\|(U,V)\|\to+\infty$ as $t\to T_{\max}^-$.

Now we want to prove that $(U(t),V(t))\in \Mah$ in its interval of definition. To this aim let us consider the $C^1$--function
\[
\rho(t)=\int_\O U^2(x,t)\,dx,
\]
which is continuous at $t=0$. By a straight calculation one can see that it verifies the initial value problem
\[
\left\{
\begin{array}{l}
 \rho'(t)=a(t)(\rho(t)-1)\\
 \rho(0)=1,
\end{array}
\right.
\]
where $a(t)=2\l(U(t),V(t))$ is a continuous function. Since the previous initial value problem admits only one solution, then $\rho(t)\equiv1$ in $[0,T_{\max})$ (and an analogous result holds for $V(t)$).
%
%
Finally, by integrating by parts (by standard regularity, $(U(t),V(t))$ belongs to $H^2$ for $t>0$) and  by using the fact that $\int_\O UU_t\,dx=\int_\O VV_t\,dx=0$, one can easily obtain
\[
\frac{d}{dt} J_\b(U(t),V(t))= \int_\O\left(U_t,V_t\right)\cdot S_\b(U,V)\,dx= - \left|S_\b(U,V)\right|^2_2\leq0.
\]
This implies
\begin{equation}\label{eq:weissler_exist}
\|(U(t),V(t))\|^2\leq 2J_\b(U(t),V(t))\leq 2J_\b(u,v)<+\infty
\end{equation}
for every $t<T_{\max}$, which provides $T_{\max}=+\infty$.
\end{proof}
\begin{rem}\label{rem:flux_control}
Given $(u,v)\in \Mah$ let $(U,V)$ be the corresponding solution of
\eqref{eq:flux}. By taking in consideration inequality \eqref{eq:weissler_exist} we see that the quantities $\|(U(t),V(t))\|$, $|(U(t),V(t))|_p$ (with $p\leq6$), $\l(U(t),V(t))$ and $\mu(U(t),V(t))$ are bounded by constants which only depend on $J_\b(u,v)$ (in particular, they are independent of $t$).
\end{rem}
\begin{lemma}\label{lemma:flux_sign_preserving_beta_finite}
Let $c\in C\left([0,T];L^{3/2}(\O)\right)$ and let $U\in
C^1\left((0,T]; L^2(\O)\right)\cap C\left([0,T]; \spc\right)$ be a
solution of
\[
\partial_t U-\Delta U=c(x,t) U,\quad U(\cdot,t)\in\spc, U(x,0)\geq0.
\]
Then $U(x,t)\geq0$ for every $t$.
\end{lemma}

\begin{proof}
Since $c:[0,T]\to L^{3/2}$ we can write $|c(x,t)|\leq k + c_1(x,t)$,
where $k$ is constant and $|c_1|_{3/2}<1/C^2_S$ (here $C_S$
denotes the Sobolev constant of the embedding $\spc\hookrightarrow
L^6(\O)$). Let
\[
\rho(t)=\frac12\int_\O |U^-(x,t)|^2\,dx.
\]
We obtain that $\rho\in C^1((0,T])\cap C([0,T])$ and $\rho(0)=0$; moreover,
\[
\begin{split}
\rho'(t) &= -\int_\O U^-\partial_t U\,dx = -\int_\O\left(U^-\Delta U + c(x,t) |U^-|^2\right)\,dx\\
      &\leq -\|U^-\|^2 + k|U^-|_{2}^2 + |c_1|_{3/2}|U^-|_{6}^2
      \leq \left(-1+C_S^2|c_1|_{3/2}\right)\|U^-\|^2 + k|U^-|_{2}^2\\
      &\leq 2k\rho(t).
\end{split}
\]
Thus we deduce that $\rho(t)\leq e^{2k}\rho(0)$ and the lemma follows.
\end{proof}
\begin{lemma}\label{lem:cont_L2_1}
Let $w\in C^1\left((0,+\infty); L^2(\O)\times
L^2(\O)\right)\cap C\left([0,+\infty); \spc\times\spc\right)$ be a solution of
\begin{equation}\label{eq:abtr_cont}
\left\{
\begin{array}{l}
 \partial_t w -\Delta w = F(w)\\
 w(0)=w_0,
\end{array}
\right.
\end{equation}
where there exists a positive constant $C$ such that
\begin{equation}\label{ineq_cont_L2_1}
\int_\O F(w)\cdot w\,dx\leq \frac{1}{2}\|w\|^2
+C|w|^2_2\quad\text{for every }t\geq 0.
\end{equation}
Then there exists a constant $C(t)$ such that
\[
|w(t)|_2\leq C(t) |w_0|_2.
\]
\end{lemma}
\begin{proof}
Let
\[
E(t)=\frac12\int_\O w^2(t)\,dx.
\]
A straightforward computation yields
\[
 E'(t) = -\int_\O |\nabla w|^2\,dx+\int_\O F(w)\cdot w\,dx
       \leq -\frac{1}{2}\|w\|^2+C |w|^2_2
       \leq 2C  E(t),
\]
from which we obtain $E(t)\leq e^{2C t} E(0)$, concluding the proof.
\end{proof}
\begin{lemma}\label{lem:cont_L2_2}
For $i=1,2$ take $(u_i,v_i)\in \Mah$ and let $(U_i(t),V_i(t))$ be the corresponding solution of \eqref{eq:flux}. There exists a constant $C$, only depending on $\max_i J_\b(u_i,v_i)$, such that, for every $t$
\begin{enumerate}
 \item $|\l(U_1(t),V_1(t))-\l(U_2(t),V_2(t))|\leq C\left(\|U_1(t)-U_2(t)\|+|V_1(t)-V_2(t)|_{2}\right)$;
 \item $|\mu(U_1(t),V_1(t))-\mu(U_2(t),V_2(t))|\leq C\left(\|V_1(t)-V_2(t)\|+|U_1(t)-U_2(t)|_{2}\right)$.
\end{enumerate}
\end{lemma}
\begin{proof}
We prove only the first inequality, since the second one is analogous. We have
\[
\begin{split}
 |\l(U_1,V_1&)-\l(U_2,V_2)| \leq \int_\O\left||\nabla U_1|^2-|\nabla U_2|^2\right|\,dx+
    \int_\O\left|U_1^4-U^4_2\right|\,dx + \b\int_\O\left|U_1^2V_1^2-U_2^2V_2^2 \right|\,dx\\
  &\leq \int_\O\left|\nabla U_1+\nabla U_2\right| \left|\nabla U_1-\nabla U_2\right|\,dx
     +\int_\O(U_1^2+U_2^2)|U_1+U_2| \left|U_1-U_2\right|\,dx +\\
  &\qquad\qquad\qquad+ \b\int_\O U_1^2|V_1+V_2|\left|V_1-V_2\right|\,dx+\b\int_\O V_2^2|U_1+U_2|
     \left|U_1-U_2\right|\,dx\\
  &\leq \|U_1+U_2\| \|U_1-U_2\| + |(U_1^2+U_2^2)(U_1+U_2)|_{2} |U_1-U_2|_{2}+ \\
  &\qquad\qquad\qquad+
  |\b U_1^2(V_1+V_2)|_{2} |V_1-V_2|_{2}+ |\b V_2^2(U_1+U_2)|_{2} |U_1-U_2|_{2},
\end{split}
\]
from which we can conclude the proof by recalling Remark \ref{rem:flux_control} and Poincar\'e's inequality.
\end{proof}
\begin{coro}\label{coro:L^2_cont_1}
For $i=1,2$ consider $(u_i,v_i)\in \Mah$ and let $(U_i(t),V_i(t))$ be the corresponding solution of \eqref{eq:flux}. There exists a constant $C=C(t)$, depending on $t$ (and also on $\max_i J_\b(u_i,v_i)$) such that
\[
|(U_1(t),V_1(t))-(U_2(t),V_2(t))|_2\leq C(t) |(u_1,v_1)-(u_2,v_2)|_2.
\]
\end{coro}
\begin{proof}
We want to apply Lemma \ref{lem:cont_L2_1} to
$ w=(w_1,w_2)=(U_1-U_2,V_1-V_2)$. Subtracting the equations
for $(U_1,V_1)$ and $(U_2,V_2)$ we end up with a system like \eqref{eq:abtr_cont}, thus we only need to check that
\[
F=\left(
\begin{array}{c}
U_2^3-U_1^3+\b(U_2V_2^2-U_1V_1^2)+\lambda(U_1,V_1)U_1-\lambda(U_2,V_2)U_2\\
V_2^3-V_1^3+\b(U_2^2V_2-U_1^2V_1)+\mu(U_1,V_1)V_1-\mu(U_2,V_2)V_2
\end{array}
\right),
\]
satisfies \eqref{ineq_cont_L2_1}. To make the calculation easier, we split $F$ into four terms, after
adding and subtracting some suitable quantities. The first term is
\[
F_1=-\left(
\begin{array}{c}
(U_1^2+U_1U_2+U_2^2)w_1\\
(V_1^2+V_1V_2+V_2^2)w_2
\end{array}
\right),
\]
from which we obtain, by recalling Remark \ref{rem:flux_control},
\[
\begin{split}
\int_\O F_1(w)\cdot w\, dx &\leq |U_1U_2|_3 |w_1|_6 |w_1|_2 +|V_1V_2|_3 |w_2|_6 |w_2|_2\\
&\leq C (\|w_1\| |w_1|_2 +\|w_2\| |w_2|_2)\\
&\leq \frac{1}{2} \left( \|w_1\|^2 + \|w_2\|^2 \right) + C' \left( |w_1|_2^2 + |w_2|_2^2 \right)
\end{split}
\]
(where in the last step we have used Young's inequality). The second term is
\[
F_2=-\b\left(
\begin{array}{c}
U_1(V_1+V_2)w_2+V_2^2w_1\\
V_1(U_1+U_2)w_1+U_2^2w_2
\end{array}
\right),
\]
which immediately gives, reasoning in the same way as above
\[
 \begin{split}
\int_\O F_2(w)\cdot w\,dx & \leq  \b(|U_1(V_1+V_2)|_3 |w_2|_6 |w_1|_2 + |V_2^2|_3 |w_1|_6 |w_1|_2 +\\
&\qquad\qquad+
|V_1(U_1+U_2)|_3 |w_1|_6 |w_2|_2 + |U_2^2|_3 |w_2|_6 |w_2|_2)  \\
 &\leq C \left[ \|w_1\| (|w_1|_2 + |w_2|_2) + \|w_2\| (|w_1|_2 + |w_2|_2) \right] \\
&\leq \frac{1}{2}\left( \|w_1\|^2 + \|w_2\|^2 \right) + C'\left( |w_1|_2^2 + |w_2|_2^2 \right).
\end{split}
\]
The third term is
\[
F_3=\left(
\begin{array}{c}
\lambda(U_1,V_1)w_1\\
\mu(U_1,V_1)w_2
\end{array}
\right),\quad\text{from which}\quad\int_\O F_3(w)\cdot w\,dx\leq C|w|_2^2
\]
(where we used again Remark \ref{rem:flux_control}). Finally, the last term is
\[
F_4=\left(
\begin{array}{c}
(\lambda(U_1,V_1)-\lambda(U_2,V_2))U_2\\
(\mu(U_1,V_1)-\mu(U_2,V_2))V_2
\end{array}
\right),
\]
which can be ruled out by using Lemma \ref{lem:cont_L2_2}. We obtain
\[
\begin{split}
\int_\O F_4(w)\cdot w\,dx  &\leq
\left(\|U_1(t)-U_2(t)\|+|V_1(t)-V_2(t)|_2\right)\int_\O
|U_2 w_1|\,dx +\\
&\qquad+C\left(\|V_1(t)-V_2(t)\|+|U_1(t)-U_2(t)|_2\right)\int_\O
|V_2 w_2|\,dx\\
&\leq \frac{1}{2} \left( \|w_1\|^2 + \|w_2\|^2 \right) + C'\left( |w_1|_2^2 + |w_2|_2^2 \right).
\end{split}
\]
Therefore $F=F_1+F_2+F_3+F_4$ satisfies \eqref{ineq_cont_L2_1}, and hence Lemma \ref{lem:cont_L2_1} yields the desired result.
\end{proof}

\begin{proof}[Proof of Proposition \ref{prop:flux_beta_finite}]
Properties (i) and (ii) have been proved in Lemma \ref{lemma:flux_sign_preserving_beta_finite} and Corollary \ref{coro:L^2_cont_1} respectively; let us prove (iii). This is a direct consequence of the estimate on the derivative of $J_\b$ expressed in Lemma \ref{lem:weissler_exist}. In fact the following holds
\[
\begin{split}
\dist\left((U(s),V(s)),(U(t),V(t))\right) & = \left| \int_s^t \partial_\tau (U(\tau),V(\tau)) d\tau \right|_2 \\
&\leq |t-s|^{1/2} \left( \int_s^t |S_\b(U(\tau),V(\tau))|_2^2 d\tau \right)^{1/2} \\
&=|t-s|^{1/2} |J_\b(U(s),V(s))-J_\b(U(t),V(t))|^{1/2}.
\qedhere\end{split}
\]
\end{proof}

We now turn to the construction of the flux $\eta_\infty$.

\begin{proof}[Proof of Lemma \ref{lem:lambda_mu_tilde}]
By definition, $\tS$ is the projection of the gradient of $J^*$ at $w$ on the tangential space of the manifold
$\left\{w\in\spc:\,(w^+,w^-)\in\Mah\right\}$ at $w$, thus
\[
\tS(w)=w+\mathcal{L}w^3-\tl\mathcal{L}w^+ + \tmu\mathcal{L}w^-,
\]
where the coefficients $\tl,\tmu$ satisfy $\int_\O w^+\tS(w)\,dx=\int_\O w^-\tS(w)\,dx=0$, that is
\[
\left(
\begin{array}{cc}
 \int_\O w^+\invD w^+\,dx & - \int_\O w^+\invD w^-\,dx\smallskip\\
  - \int_\O w^-\invD w^+\,dx & \int_\O w^-\invD w^-\,dx
\end{array}
\right)
\left(
\begin{array}{c}
 \tl\smallskip\\
 \tmu
\end{array}
\right)
=
\left(
\begin{array}{c}
 \int_\O \left(w+\invD w^3\right)w^+\,dx\smallskip\\
 -\int_\O \left(w+\invD w^3\right)w^-\,dx
\end{array}
\right).
\]
Denoting by $A$ the coefficient matrix, we compute\footnote{By using the identity $\int_\O f\invD g\,dx=\int_\O \nabla\invD f\cdot\nabla\invD g\,dx$.}
\[
\det A = \left(\int_\O|\nabla\invD w^+|^2\,dx\right)\left(\int_\O|\nabla\invD w^-|^2\,dx\right)-
\left(\int_\O\nabla\invD w^+\cdot\nabla\invD w^-\,dx\right)^2\geq0,
\]
by H\"older inequality, and $\det A=0$ if and only if $a\nabla\invD w^+ +
b\nabla\invD w^-\equiv0$, for some $a,b$ not both zero. But this would imply that the $\spc$--function
$\invD (aw^+ + b w^-)$ would have an identically zero gradient and therefore $aw^++bw^-\equiv0$, in contradiction with the fact that, by assumption, $|aw^+ + b w^-|_2^2\geq(a^2+b^2)R_1^2$. Thus the $L^2$--continuous function $\det A$ is strictly positive on the $L^2$--compact set $\left\{w:\,|w^\pm|_2\geq R_1,\,\|w\|\leq R_2\right\}$, i.e. it is larger than a strictly positive constant (only depending on $R_1,R_2$). This provides (existence, uniqueness and)
an explicit expression of $\tl(w)$ and $\tmu(w)$ for any $w$ satisfying the previous assumptions. The regularity of these functions descends from such explicit expressions, once one notices that they are both products of Lipschitz continuous functions (and therefore bounded when $\|w\|\leq R_2$). Just as an example, we prove the Lipschitz continuity of the term $\int_\O w^+\invD w^3\,dx$. We have\footnote{Remember that, by standard elliptic regularity results, both $\invD:L^2 \to L^2$ and $\invD:L^{6/5} \to L^6$ are continuous.}
\[
\begin{split}
\left|\int_\O w_1^+\invD w_1^3\,dx-\int_\O w_2^+\invD w_2^3\,dx\right|
 &\leq \int_\O\left|w_1^+-w_2^+\right|\invD w_1^3\,dx +
  \int_\O w_2^+\left|\invD (w_1^3- w_2^3)\right|\,dx\\
 &\leq C\left|w_1^+-w_2^+\right|_2\left|w_1^3\right|_2
   +\left|w_2^+\right|_{6/5}\left|\invD(w_1^3- w_2^3)\right|_6\\
  &\leq CR_2^3 \left|w_1-w_2\right|_2+C R_2 \left|w_1^3- w_2^3\right|_{6/5}\\
&\leq C R_2^3 \left|w_1-w_2\right|_2.
\end{split}
\]
All the other terms can be treated the same way.
\end{proof}

\begin{rem}\label{rem:lip_cont_s_tild}
By reasoning as in the end of the previous proof, it can be proved that, whenever $w_1,w_2$ belong to the set
\[
\left\{w\in\spc:\,|w^\pm|_2\geq R_1,\,\|w\|\leq R_2\right\},
\]
there exists a constant $L$, only depending on $R_1,R_2$, such that
\[
\left|\tS(w_1)-\tS(w_2)\right|_2\leq L\left|w_1-w_2\right|_2,
\]
\[
\left\|\tS(w_1)-\tS(w_2)\right\|\leq L\left\|w_1-w_2\right\|.
\]
\end{rem}

\begin{proof}[Proof of Lemma \ref{lem:flux_infty_exist}]
Let us fix $0<R_1<1$ and $R_2> 2(c_\infty+1)$. By Remark \ref{rem:lip_cont_s_tild} we have that $-\tS$, as a map from $\spc$ into itself, is $H^1_0$--Lipschitz continuous on the mentioned set, with Lipschitz constant only depending on $R_1,R_2$; we infer existence (and uniqueness) of a maximal solution of the Cauchy problem, defined on $[0,T_{\max})$. Moreover, for any $t\in(0,T_{\max})$, we have
\[
\frac{d}{dt}|W^\pm(t)|_2^2=\pm 2\int_\O W^\pm W_t\,dx = \mp 2\int_\O W^\pm\tS(W)\,dx=0
\]
(by Lemma \ref{lem:lambda_mu_tilde}), and
\[
\begin{split}
\frac{d}{dt}J_\infty(W^+(t),W^-(t)) &= \frac{d}{dt}\int_{\O} \left(\frac12|\nabla W|^2+\frac14 W^4\right)\,dx=
 \int_\O \left(-\D W + W^3\right)W_t\,dx =\\
&=\int_\O -\D\left(W + \invD W^3\right) W_t\,dx = \\
&= \int_\O \nabla \left(\tS(W) + \invD(\tl W^+ -\tmu W^-) \right)\cdot \nabla \left( -\tS(W)\right)\,dx\\
& = -\|\tS(W(t))\|^2.
\end{split}
\]
Thus, for any $t\in(0,T_{\max})$, we obtain $|W^\pm(t)|_2=1>R_1$ and $\|W(t)\|\leq 2 J_\infty(W^+(t),W^-(t))\leq
2J_\infty(u,v)<R_2$. In particular this implies $T_{\max}=+\infty$, concluding the proof of the lemma.
\end{proof}
\begin{proof}[Proof of Proposition \ref{prop:flux_beta_infinite}]
(i) Consider $(u_1,v_1),(u_2,v_2) \in \Mah_\infty^{c_\infty+1}$ and let $W_1(t),W_2(t)$ be the corresponding solutions of \eqref{eq:flux_infty}. We notice first of all that Remark \ref{rem:lip_cont_s_tild} applies, providing
the existence of $L=L(c_\infty)$ such that
\[
\frac{d}{dt}\left|W_1(t)-W_2(t)\right|_2^2\leq 2L\, \left|W_1(t)-W_2(t)\right|_2^2,
\]
which implies
\[
\left|W_1(t)-W_2(t)\right|_2^2\leq e^{2Lt}\left|W_1(0)-W_2(0)\right|_2^2.
\]
Therefore
\[
\begin{split}
\dist^2((W_1^+(1),W_1^-(1)),(W_2^+(1),W_2^-(1)))&\leq  |W_1(1)-W_2(1)|_2^2\\
&\leq e^{2L}\left|W_1(0)-W_2(0)\right|_2^2\\
&\leq 2e^{2L}(|u_1-v_1|_2^2+|u_2-v_2|_2^2). \qedhere
\end{split}
\]
(ii) Notice first that
\[
\dist^2((W^+(s),W^-(s)),(W^+(t),W^-(t)))\leq \left|W(s)-W(t)\right|_2^2.
\]
Now, Lemma \ref{lem:flux_infty_exist} allows to compute
\[
\begin{split}
|W(s)-W(t)|_2 & \leq C_S \|W(s)-W(t)\|= C_S \left\|\int_s^t \partial_\tau W(\tau) d\tau\right\| \\
&\leq C_S |t-s|^{1/2} \left( \int_s^t \|\tS(W(\tau))\|^2 d\tau \right)^{1/2} \\
&=C_S |t-s|^{1/2} |J_\infty(W^+(s),W^-(s))-J_\infty(W^+(t),W^-(t))|^{1/2},
\end{split}
\]
and the two inequalities together conclude the proof.
\end{proof}

\subsection*{Acknowledgements}

Work partially supported by MIUR, Project ``Metodi
Variazionali ed Equazioni Differenziali Non Lineari''.
The second author was supported by FCT (grant SFRH/BD/28964/2006).


\noindent\verb"benedettanoris@gmail.com"\\
\verb"susanna.terracini@unimib.it"\\
Dipartimento di Matematica e Applicazioni, Universit\`a degli Studi
di Milano-Bicocca, via Bicocca degli Arcimboldi 8, 20126 Milano,
Italy

\noindent \verb"htavares@ptmat.fc.ul.pt"\\
University of Lisbon, CMAF, Faculty of Science, Av. Prof. Gama Pinto
2, 1649-003 Lisboa, Portugal

\noindent \verb"gianmaria.verzini@polimi.it"\\
Dipartimento di Matematica, Politecnico di Milano, p.za Leonardo da
Vinci 32,  20133 Milano, Italy

\end{document}